\documentclass[a4paper, 12pt]{amsart}

\pdfoutput = 1

\usepackage[english]{babel}
\usepackage[utf8]{inputenc}
\usepackage[T1]{fontenc}
\usepackage{csquotes}
\usepackage{verbatim}

\usepackage{amsfonts}
\usepackage{amsmath}
\usepackage{amsthm}
\usepackage{graphicx}
\usepackage{amssymb}
\usepackage{mathtools, xcolor}
\usepackage[colorinlistoftodos]{todonotes}
\usepackage[breaklinks]{hyperref}
\hypersetup{
colorlinks = true, % Colours links instead of ugly boxes
 linkcolor = blue, % Colour of internal links
}

\usepackage[a4paper,top=2cm,bottom=2cm,left=2cm,right=2cm,marginparwidth=2cm]{geometry}

\DeclarePairedDelimiter\floor{\lfloor}{\rfloor}

\begin{document}
%\theoremstyle{theorem}
%\numberwithin{equation}{section}
\theoremstyle{plain}
\newtheorem{theorem}{Theorem}[]
\newtheorem*{theorem*}{Theorem}
\newtheorem{lemma}{Lemma}
\newtheorem{corollary}{Corollary}
\newtheorem{proposition}{Proposition}
\newtheorem{conjecture}{Conjecture}

\theoremstyle{definition}
\newtheorem{definition}{Definition}
\newtheorem{remark}{Remark}

\def\eps{\varepsilon}
\def\F{\mathbb {F}}
\def\C{\mathbb {C}}
\def\K{\mathbb {K}}
\def\A{{\mathcal A}}
\def \I{\mathcal{I}}

\author{Alexandr Grebennikov, Arsenii Sagdeev, \\ Aliaksei Semchankau, Aliaksei Vasilevskii}
\title{On the sequence $n! \bmod p$}

\maketitle

\begin{abstract}
%% No indent
\setlength\parindent{0pt}
We prove, that the sequence $1!, 2!, 3!, \dots$ produces at least $(\sqrt{2} + o(1))\sqrt{p}$ distinct residues modulo prime $p$.
Moreover, factorials on an interval $\I \subseteq \{0, 1, \dots, p - 1\}$ of length $N > p^{7/8 + \eps}$ produce at least $(1 + o(1))\sqrt{p}$ distinct residues modulo $p$.
As a corollary, we prove that every non-zero residue class can be expressed as a product of seven factorials $n_1! \dots n_7!$ modulo $p$, where  $n_i = O(p^{6/7+\varepsilon})$ for all $i=1,\dots,7$, which provides a polynomial improvement upon the preceding results.

\end{abstract}

\tableofcontents

\section{Introduction}

Wilson's theorem represents one of the most elegant results in elementary number theory. It states that if $p$ is  a prime number, then $(p-1)! = -1 \bmod p$. As one of its simple corollaries, we note that $(p-2)! = 1! \bmod p$, and thus not all the residues from
$$\A(p) \coloneqq \{i! \bmod p: i \in [p-1]\}$$
are distinct. Erd\H{o}s conjectured \cite{RockSchin}, that this is not the only coincidence, i.e., that $|\A(p)| < p-2$. Surprisingly, despite the long history of this natural problem, Erd\H{o}s' conjecture remains widely open though verified \cite{Trud} for all primes  $p < 10^9$.

At the same time, it is widely believed (see \cite{BrougBar, CVZ} and \cite{guy}, \textbf{F11}) that the elements of $\A(p)$ may be considered as more or less `independent uniform random variables' for large $p$. In particular, it is conjectured that
\begin{equation*}
	|\A(p)| = \left(1-\frac{1}{e}+o(1)\right)p 
\end{equation*}
as $p \rightarrow \infty$. However, the best lower bound up to now is due to Garc\'ia~\cite{garcia2007}:
\begin{theorem*}[\bf Garc\'ia]
    $$|\A(p)| \geqslant\left( \sqrt{\frac{41}{24}}+o(1)\right) \sqrt{p}.$$
\end{theorem*}
The strategy in \cite{garcia2007} was to prove that $\A(p)\A(p)$ contains residues with certain properties, which forces the estimate $|\A(p)\A(p)| \geqslant (41/48 + o(1))p$ to hold; combined with the observation 
$$
 {|\A(p)| + 1 \choose 2} \geqslant |\A(p)\A(p)|
$$ 
this yields the result. 
We improve it to the following:
\begin{theorem}\label{thm:product}
$$
    |\A(p)\A(p)| \geqslant p + O(p^{13/14}(\log{p})^{4/7}).
$$
\end{theorem}
\begin{corollary}
    $$
    |\A(p)| \geqslant \big( \sqrt{2} + o(1)\big)\sqrt{p}.
    $$
\end{corollary}

One of the natural ways to generalize this problem is to consider it in a `short interval' setting (see \cite{garaev-hernandez,GLS04,KlurMun,Lev}). {\it Throughout this paper}, let $p$ be a large enough prime and $L,N$ be integers such that $0 < L+1 < L+N < p$. Following Garaev and Hern\'andez \cite{garaev-hernandez}, we define a `short interval' analogue of $\A(p)$ as follows:
$$ \A(L, N) \coloneqq \{n! \bmod p: L + 1 \leqslant n \leqslant L + N\}. $$
As $L$ will not play any role, we write $\A_N$ for short.
To bound the cardinality of this set from below, it is usually fruitful to estimate the size of 
%$\frac{\A(L, N)}{\A(L, N)}$, 
$\A_N/\A_N$,
the set of pairwise fractions, since we trivially have 
%$|\A(L,N)|^2 \geqslant\left|\frac{\A(L, N)}{\A(L, N)}\right|$. 
$|\A_N|^2 \geqslant |\A_N/\A_N|$.
The first lower bounds on the size of this set of fractions were linear on $N$  (see \cite{GLS04,KlurMun}), while Garaev and Hern\'andez \cite{garaev-hernandez} found the following logarithmic improvement.

\begin{theorem*}[\bf Garaev-Hern\'andez]
	Let $p^{1/2+\varepsilon}<N<p/10$. Then
	\begin{equation*}
		%\left|\frac{\A(L, N)}{\A(L, N)}\right| 
		\left| \A_N/\A_N\right|
		\geqslant 
		c_0 N\log\left(\frac{p}{N}\right) 
	\end{equation*}
	for some $c_0=c_0(\varepsilon)>0$.
\end{theorem*}
The strategy in \cite{garaev-hernandez} was to observe $\A_N/\A_N$ to contain the sets $X_1, X_2, \dots, X_M$ defined as $X_j = \{(x + 1)(x + 2)\dots (x+j) : L + 1 \leqslant x \leqslant L + N - M\}$, and then prove $X_j$'s to be `large', but their intersections $X_k \cap X_j$ to be `small', which makes inclusion-exclusion formula applicable: 
$$
|\A_N/\A_N| \geqslant |X_1 \cup X_2 \cup \dots| \geqslant \sum_{j}|X_j| - \sum_{k < j}|X_k \cap X_j| \gg \sum_{j}|X_j|.
$$

In the present paper we give the following improvement of this result.
\begin{theorem}\label{thm:n-fact}
Let $N$ be such that $c_5 \sqrt{p}(\log{p})^2 \leqslant N \leqslant p$. 
Let $K := \frac{p}{N}, Q := \frac{N}{\sqrt{p}(\log{p})^2}$. Then
$$
    %\left| \frac{\A(L, N)}{\A(L, N)} \right| 
    |\A_N/\A_N|
    \geqslant
    \begin{cases}
    p + O(p^{13/14}(\log{p})^{4/7}) &  \mbox{ if } N \geqslant c_1 p^{13/14}(\log{p})^{4/7}, \\
    p + O(p^{5/6}K^{4/3}(\log{p})^{4/3}) & \mbox{ if } c_1 p^{13/14}(\log{p})^{4/7} \geqslant N \geqslant c_2 p^{7/8}\log{p}, \\
    cNQ^{1/3}(\log{Q})^{-2/3} & \mbox { if } c_2 p^{7/8}\log{p} \geqslant N \geqslant c_3 p^{4/5}(\log{p})^{8/5}, \\
    cNK^{1/2} & \mbox { if } c_3 p^{4/5}(\log{p})^{8/5} \geqslant N \geqslant c_4 p^{4/5}(\log{p})^{4/5}, \\
    cNQ^{1/3} & \mbox { if } c_4 p^{4/5}(\log{p})^{4/5} \geqslant N \geqslant c_5 p^{1/2}(\log{p})^{2}.
    \end{cases}
$$
where $c, c_1, c_2, c_3, c_4, c_5 > 0$ are some absolute constants, whose values can be extracted from the proof.
\end{theorem}

\begin{corollary}
    For $N \gg p^{7/8}\log{p}$,
        $$
        |\A_N| \geqslant (1 + o(1))\sqrt{p}.
        $$
\end{corollary}
To derive it, we continue the strategy from \cite{garaev-hernandez} as follows: using strong results from Algebraic Geometry, we prove `best possible' bounds $|X_j| \geqslant (1 + o(1))N$ and $|X_k \cap X_j| \leqslant (1 + o(1))N^2/p$ for prime $k, j$. Then we observe, that bounds on sets $X_j$ and their intersections imply they behave `too independently', and therefore the size of their union is at least $p + o(p)$ (see Lemma \ref{lem:set_intersect}), which implies that $\A_N/\A_N$ has size at least $p + o(p)$. 

This strategy turns out to be helpful when proving Theorem \ref{thm:product} as well.

One of the nice applications of these results deals with representation of the residues as a product of several factorials. It is not hard to see that the classical Wilson's theorem implies the following. Any given $a \in [p-1]$ can be represented\footnote{Indeed, one may easily verify that, depending on the `parity' of the inverse residue $b \equiv a^{-1}$, we have either $a \equiv (b-1)!(p-1-b)!$, or $a \equiv -(b-1)!(p-1-b)! \equiv (b-1)!(p-1-b)!(p-1)!$ modulo~$p$.} as a product of three factorials
\begin{equation*}
	a \equiv n_1!n_2!n_3! \bmod p
\end{equation*}
for some $n_1,n_2, n_3 \in [p-1]$. The aforementioned conjecture on the `randomness' of $\A(p)$ implies that even two factorials are enough. However, if we add an additional constraint the all $n_i$ should be of the magnitude $o(p)$ as $p \rightarrow \infty$, it becomes not so clear how many factorials are required. Garaev, Luca, and Shparlinski \cite{GLS04} coped with seven.

\begin{theorem*}[\bf Garaev, Luca, and Shparlinski]
	Fix any positive $\varepsilon < 1/12$. Then for all prime $p$, every residue class $a \not\equiv 0  \bmod p$ can be represented as a product of seven factorials,
	$$a \equiv n_1!\dots n_7! \pmod p,$$
	such that $n_0 \coloneqq \max_{1 \le i \le 7}{n_i} = O(p^{11/12+\varepsilon})$ as $p \rightarrow \infty$.
\end{theorem*}

During the last two decades, the number of multipliers from the last theorem was not reduced even to $6$. However, there were certain improvements on the value of $n_0$. Garc\'ia \cite{garcia2008} showed that the Theorem above holds with $n_0 = O(p^{11/12}\log^{1/2}p)$, while Garaev and Hern\'andez \cite{garaev-hernandez} relaxed it to $O(p^{11/12}\log^{-1/2}p)$. Since our new Theorem~\ref{thm:n-fact} improves the bounds used in the latter proof, one can obtain a slight (again, {\it polynomial}) improvement on the value of $n_0$ by following the same proof.

\begin{theorem}
	Fix any positive $\varepsilon < 1/7$. Then for all prime $p$, every residue class $a \not\equiv 0  \bmod p$ can be represented as a product of seven factorials,
	$$a \equiv n_1!\dots n_7! \pmod p,$$
	such that $n_0 \coloneqq \max_{1 \le i \le 7}{n_i} = O(p^{6/7+\varepsilon})$ as $p \rightarrow \infty$.
\end{theorem}

The remainder of the text has the following structure. In Section 2 we introduce some notations and useful lemmas, in Section 3 we prove results on images of `generic' polynomials, in Section 4 we apply these results to polynomials $P_j(x) = (x+1)\ldots (x+j)$, and, finally, in Sections 5 and 6 wee prove theorems \ref{thm:product} and \ref{thm:n-fact}.

\section{Conventions and Preliminary Results}
Here and below, $p$ denotes a large prime number.

Whenever $A$ is a set, we identify it with its indicator function, meaning
$$
A(x) = 
\begin{cases}
    1,\ x \in A, \\
    0,\ x \not\in A.
\end{cases}
$$

Throughout the paper, the standard notation $\ll, \gg$, and respectively $O$ and $\Omega$ is applied to positive quantities in the usual way. That is, $X \ll Y, Y \gg X, X = O(Y)$ and $Y = \Omega(X)$ all mean that $Y \geqslant cX$, for some absolute constant $c > 0$.  \\

A polynomial $f \in \F_p[x]$ is \emph{decomposable}, if $f = g \circ h$ for some polynomials $g, h \in \F_p[x]$ of degrees at least $2$. Otherwise it is  \emph{indecomposable}.

We recall that for any integer $d > 0$ and $a \in \F_p$, the \emph{Dickson polynomial} $D_{d,a} \in \F_p[x]$ is defined to be
the unique polynomial such that
$D_{d,a}(x + \frac{a}{x}) = x^d + (\frac{a}{x})^d$. There is also an explicit formula for it:
$$
    D_{d, a}(x) = \sum_{i = 0}^{\floor{d/2}}\frac{d}{d-i}{{d-i}\choose i} (-a)^i x^{d-2i}. 
$$

For a positive integer $j$ define the polynomial 
$$
P_j(x) = \prod_{i=1}^{j}(x + i).
$$

Given a set $A$ and a polynomial $P \in \F_p[x]$, denote by $P(A)$ the set $\{P(a)\ (\bmod\ p) : a \in A\}$.

A key lemma to estimate the union of sets:
\begin{lemma}\label{lem:set_intersect}
Let $A_1, A_2, \dots, A_n$ be finite sets, and let $a \geqslant b$ be positive integers, such that the properties hold:
\begin{itemize}
    \item $|A_i| \geqslant a$ for all $i$,
    \item $|A_i \cap A_j| \leqslant b$ for all $i \neq j$.
\end{itemize}
Let $A := A_1 \cup A_2 \cup \dots A_n$. Then
$$
|A| \geqslant \frac{a^2}{b}\bigg( 1 - \frac{a}{nb}\bigg).
$$
\end{lemma}
\begin{proof}
    Let $S = \sum_{i\leqslant n}\sum_{a \in A}A_i(a) \geqslant na$. 
    Observe that 
    $$
   S^2 = \bigg(
    \sum_{a \in A}\big( \sum_{i \leqslant n} A_i(a)  \big)
    \bigg)^2
    \leqslant 
    |A|\sum_{a \in A}\big( \sum_{i \leqslant n}A_i(a)\big)^2
    = |A|\sum_{a \in A}\sum_{i, j \leqslant n}A_i(a)A_j(a) 
    = 
    $$
    $$
    = |A|\sum_{i, j \leqslant n}|A_i \cap A_j|
    \leqslant
    |A|\big(S + (n^2 - n)b\big) ,
    $$
    which implies
    $$
    |A| \geqslant \frac{S^2}{S + (n^2 - n)b} \geqslant \frac{(na)^2}{na + (n^2 - n)b} \geqslant 
    \frac{na^2}{a + nb} = \frac{a^2}{b}\frac{1}{1 + \frac{a}{bn}} \geqslant
    \frac{a^2}{b}\bigg(1 - \frac{a}{bn}\bigg).
    $$
\end{proof}
\begin{comment}
\begin{proof}
    Let $S = \sum_{i}\sum_{a \in A}A_i(a) \geqslant n\alpha|A|$. 
    Observe that 
    $$
    n^2\alpha^2 |A| 
    \leqslant S^2 
    = \bigg(
    \sum_{a \in A}\big( \sum_{i} A_i(a)  \big)
    \bigg)^2
    \leqslant 
    |A|\sum_{a \in A}\big( \sum_{i}A_i(a)\big)^2
    = |A|\sum_{a \in A, i, j}A_i(a)A_j(a) 
    = 
    $$
    $$
    = |A|\sum_{i, j}|A_i \cap A_j|
    \leqslant 
    |A|\bigg(S + (1 - \eps)\sum_{i \neq j}\frac{|A_i||A_j|}{|A|}\bigg)
    \leqslant 
    |A|\bigg(S + (1-\eps)\frac{S^2}{|A|}\bigg),
    $$
    which implies
    $$
    \eps S \leqslant |A|.
    $$
    Since $S \geqslant n\alpha |A|$ we obtain $n \leqslant 1/\alpha \eps$, which completes the proof.
\end{proof}
\end{comment}

\section{On images of generic polynomials}

The two following results seem to be well-known, yet not explicitly written in the literature (see \cite{chang}, \cite{cilleruelo-garaev-ostafe-shparlinskii} for more information on related questions); we prove them in here for the sake of transparency.

\begin{lemma}\label{lem:xj}
    Let $P \in \F_p[x]$ of degree $d$ be such that $ \frac{P(x) - P(y)}{x - y}$ is absolutely irreducible over $\F_p$, and let $\I$ be an arithmetical progression in $\F_p$, then:
    $$
    |P(\I)| = |\I| + O\big(|\I|^2 p^{-1} + d^2\sqrt{p}(\log{p})^2\big).
    $$
\end{lemma}

\begin{lemma}\label{lem:xj_xk}
    Let $P, Q \in \F_p[x]$ of maximal degree $d$ be such that $P(x) - Q(y)$ is absolutely irreducible over $\F_p$, and let $\I$ be an arithmetical progression in $\F_p$, then:
    $$
    |P(\I) \cap Q(\I)| \leqslant |\I|^2 p^{-1} + O(d^2\sqrt{p}(\log{p})^2).
    $$
\end{lemma}

We postpone their proofs until the end of the section, and formulate some helpful results, which are only to be used in this section.

Given $P, Q \in \F_p[x]$, let us define $\phi(P, Q) \in \F_p[x, y]$ as
$$
    \phi(P, Q)(x, y) :=
        \begin{cases}
            P(x) - Q(y), \text{ if } P \neq Q,\\ \\
            \frac{P(x) - P(y)}{x - y},\ \ \ \ \ \text{ if } P = Q.
        \end{cases}
$$
Let us also define
$$
    J(P, Q) := \#\{(x, y) \in \F_p \times \F_p : \phi(P, Q)(x, y) = 0\}.
$$

\begin{lemma}\label{lem:qj}
Given $P, Q \in \F_p[x]$, suppose that $\phi(P, Q)$ is absolutely irreducible over $\F_p$. Then
$$
J(P, Q) = p + O(d^2 \sqrt{p}),
$$
where $d$ is a degree of $\phi(P, Q)$.
\end{lemma}
\begin{proof}

We recall the modification of classical Lang-Weil result \cite{lang-weil}, with error term due to Aubry and Perret \cite{aubry-perret}:
\begin{theorem*}[\bf Lang-Weil]%\label{thm:lang-weil}
Let $\F_q$ be a finite field. Let $X \subseteq \mathbb{A}_{\F_q}^2$
be a geometrically irreducible
hypersurface of degree $d$. 
Then
$$
|X(\F_q) - q| \leqslant (d-1)(d-2)\sqrt{q} + d - 1.
$$
\end{theorem*}

Since $\phi(P, Q)(x, y)$ is absolutely irreducible over $\F_p$, its set of zeros is (by definition) a geometrically irreducible hypersurface, and therefore the Lang-Weil Theorem is applicable.

This implies the conclusion of the lemma.
\end{proof}

Given a subset $\I \subseteq \F_p$, let us define 
$$
J_{\I}(P, Q) := \#\{(x, y) \in \I \times \I : \phi(P, Q)(x, y) = 0\}.
$$

We need the following lemma, proof of which is already contained in \cite{garaev-hernandez}  but we write it down in full generality for explicity.
\begin{lemma} \label{intersection}
Let $P, Q \in \F_p[x]$ be such that $\phi(P, Q)$ has no linear divisors. Let $\I$ be an arithmetical progression in $\F_p$. Then
$$
J_{\I}(P, Q) = 
\frac{|\I|^2}{p^2}J(P, Q) + O(d^2\sqrt{p}(\log{p})^2),
$$
where $d$ is a degree of $\phi(P, Q)$.
\end{lemma}
\begin{proof}

We recall the statement of Lemma 1 in \cite{garaev-hernandez} (originated in \cite{chalk-smith}):
\begin{theorem*}[\bf Bombieri, Chalk-Smith]%\label{uniformity}
Let $(b_1, b_2) \in \F_p \times \F_p$ be a nonzero vector and let $f (x, y) \in \F_p[x, y]$ be a polynomial of degree $d \geqslant 1$ with the following property: there is no $c \in \F_p$ for which the polynomial $f (x, y)$ is divisible by $b_1x + b_2 y + c$. Then
$$
\bigg|\sum_{\substack{(x, y) \in \F_p \times \F_p: \\ f (x,y)=0}} e^{2\pi i(b_1x+b_2 y)/p}\bigg| \leqslant 2d^2 p^{1/2}.
$$
\end{theorem*}

In what follows, we will need a bit of Discrete Fourier Transform in $\F_p$.
Given function $f : \F_p \rightarrow \C$ define its Discrete Fourier Transform $\hat{f} : \F_p \rightarrow \C$ by
$$
\hat{f} (r) = \sum_{x \in \F_p}f(x) e^{-2\pi i\frac{rx}{p}}.
$$
One can easily verify the Fourier Inverse Transform formula:
$$
f(x) = \frac{1}{p}\sum_{r \in \F_p}\hat{f}(r)e^{2 \pi i \frac{rx}{p}}.
$$
We also need the following well-known result. 
Let $\I$ be a (finite) arithmetic progression in $\F_p$. Then
$$
\sum_{r \in \F_p} |\hat{\I}(r)| \ll p\log{p},
$$
where $\I : \F_p \rightarrow \C$ is interpreted as  characteristic function of the set $\I \subseteq \F_p$.

\medskip

Let us consider $\I$ as a characteristic functions of a set. Then 
$$
J_{\I} (P, Q) = 
\sum_{\substack{(x, y) \in \F_p \times \F_p: \\ \phi(P, Q)(x, y) = 0}} \I(x) \I(y) = 
\sum_{\substack{(x, y) \in \F_p \times \F_p: \\ \phi(P, Q)(x, y) = 0}}\frac{1}{p^2}\sum_{r_1, r_2 \in \F_p}\hat{\I}(r_1)\hat{\I}(r_2)e^{2\pi \frac{(r_1x + r_2y)}{p}} = 
$$
$$
=
\frac{|\I||\I|}{p^2}J(P, Q) + 
\frac{1}{p^2}\sum_{(r_1, r_2) \neq (0, 0)}\hat{\I}(r_1)\hat{\I}(r_2)
\sum_{\substack{(x, y) \in \F_p \times \F_p \\ \phi(P, Q)(x, y) = 0}}e^{2\pi i\frac{(r_1x + r_2y)}{p}}.
$$
Last summand might be bounded as 
$$
\frac{1}{p^2}
\sum_{r_1 \in \F_p}|\hat{\I}(r_1)|
\sum_{r_2 \in \F_p}|\hat{\I}(r_2)|
\max_{(r_1, r_2) \neq 0} 
\bigg| \sum_{\substack{(x, y) \in \F_p \times \F_p: \\ \phi(P, Q)(x, y) = 0}} e^{2\pi i \frac{r_1 x + r_2 y}{p}} \bigg| \ll (\log{p})^2 \sqrt{p}d^2.
$$
This completes the proof.
\end{proof}

Now, let us turn to the proof of the Lemma \ref{lem:xj}:
\begin{proof}
    Clearly, $|P(\I)| \leqslant |\I|$. Now let us obtain a lower bound.
    Cauchy-Bunyakovsky-Schwarz inequality implies:
    $$
    \#\{(x, y) \in \I\times \I: P(x) = P(y)\}|P(\I)| \geqslant |\I|^2,
    $$
    Clearly, 
    $$
    \#\{(x, y) \in \I \times \I: P(x) = P(y)\} 
    = |\I| + J_{\I}(P, P) 
    \leqslant |\I| + |\I|^2p^{-1} + O(d^2\sqrt{p}\log^{2}{p}),
    $$
    where we applied Lemmas \ref{intersection} and \ref{lem:qj}. Deriving the lower bound on $|P(\I)|$ completes the proof.
\end{proof}

Now we prove Lemma \ref{lem:xj_xk}:
\begin{proof}
    By Lemmas \ref{intersection} and \ref{lem:qj}:
    $$
    |P(\I) \cap Q(\I)| \leqslant 
    J_{\I}(P, Q) = \frac{|\I|^2}{p^2}J(P, Q) + O(d^2\sqrt{p}\log^{2}{p}) \leqslant 
    \frac{|\I|^2}{p} + O(d^2\sqrt{p}\log^{2}{p}).
    $$  
\end{proof}

\section{Properties of polynomials $P_j$} {

Let us deduce the following simple lemma:
\begin{lemma}\label{lem:indecomposable}
For given $5 \leqslant j < p$, the polynomial $P_j(x) \in \F_p[x]$ is not equal to $\alpha D_{j, a}(x+b)+c $ for $\alpha, a, b, c \in \F_p$. Moreover, if $j$ is prime, then $P_j(x)$ is indecomposable.  
\end{lemma}
\begin{proof}
    The second assertion is clear since $\deg P_j = j$. The first assertion can be proved by straightforward comparison of the first five leading coefficients of these two polynomials.
\end{proof}

For given $k, j$ (possibly equal) we define the polynomial $Q_{kj}(x, y)$, equal to $P_k(x) - P_j(y)$, divided by all possible linear factors. If $k = j$, we denote this polynomial by $Q_j(x, y)$. 
One can show that for $k, j < p-2$
$$
Q_{kj}(x, y)=
   \begin{cases}
        P_k(x) - P_j(y) & \mbox{if } j \neq k, \\ \\
        \frac{P_j(x) - P_j(y)}{x - y} & \mbox{if } k = j,\ j \mbox{ is odd},\\ \\
        \frac{P_j(x) - P_j(y)}{(x - y)(x + y - j - 1)} & \mbox{if } k = j,\ j \mbox { is even}.\\
    \end{cases}
$$

\begin{lemma}\label{lem:abs-irr}
$Q_{kj}(x, y)$ is absolutely irreducible over $\F_p$ for (possibly equal) primes $2 < j, k < p - 2$.   
\end{lemma}
\begin{proof}
    First, consider the case $j = k$.
    Recall a Theorem of Fried \cite{schur-conjecture}, with modification by Turnwald \cite{turnwald}. We adopt it for the field $\F_p$ and polynomial $f$ of degree less than $p$:
\begin{theorem*}[\bf Fried-Turnwald]%\label{thm:abs-irr}
    Let $f \in \F_p[x]$ be a polynomial of degree $n$, $4 < n < p$. Consider the polynomial
    $$
        \phi(x, y) := \frac{f(x) - f(y)}{x - y}
    $$
    If $f$ is indecomposable, and is not equal $\alpha D_{n,a}(x+b)+c $ for some $\alpha, a, b, c \in \F_p$, then $\phi(x, y)$ is absolutely irreducible.
\end{theorem*}
Application to the polynomial $P_j$ (along with the Lemma \ref{lem:indecomposable}), with the explicit check for $j=3$, gives the result.

\medskip

Now, consider the case $j \neq k$. 
Recall the statement of Theorem 1B in \cite{Schmidt}:
\begin{theorem*}[\bf Schmidt] %\label{thm:abs-irr-diff}
    Let 
    $$
        f(x, y) = g_0 y^d + g_1(x) y^{d-1} + \ldots + g_d(x),
    $$
    be a polynomial from $\K[x, y]$ for some field $\K$, where $g_0$ is a non-zero constant.
    Denote 
    $$
        \psi(f) = \max_{1 \le i \le d}{\frac{\deg g_i}{i}}
    $$
    and suppose $\psi(f) = \frac{m}{d}$ where $m$ is coprime to $d$. Then $f(x, y)$ is absolutely irreducible.
\end{theorem*}
Notice that $\psi(Q_{kj}) = \frac{k}{j}$, and therefore this gives the result.
\end{proof}

Clearly, if $j > k$ are odd primes, Lemma \ref{lem:abs-irr} is applicable, and Lemmas \ref{lem:xj}, \ref{lem:xj_xk} imply the following:
\begin{equation}\label{P_image}
    |P_j(\I)| = |\I| + O\big(|\I|^2 p^{-1} + j^2\sqrt{p}(\log{p})^2\big),
\end{equation}
\begin{equation}\label{PP_image}
    |P_j(\I) \cap P_k(\I)| \leqslant |\I|^2 p^{-1} + O(j^2\sqrt{p}(\log{p})^2),
\end{equation}
where $\I$ is a finite arithmetic progression in $\F_p$.
}

\section{On inequality $|\A(p)\A(p)| \geqslant p + o(p)$}
Now we prove Theorem \ref{thm:product}:
\begin{proof}
Let $\eps_1, \eps_2 > 0$ be dependent on $p$, but separated from zero. Set
    $$
    N := \floor{p^{1 - \eps_1}},\ \ 
    M := \floor{p^{\eps_2}},\ \ 
    \kappa := \log\log{p}/\log{p},\ \ 
    \delta := \min(\eps_1, 1/2 - 2\eps_1 - 2\eps_2 - 2\kappa, \eps_2 - \eps_1 - \kappa) > 0.
    $$

Let $\I$ be the set of odd numbers, not exceeding $2N - M$, and let $Y_j := P_j(\I)$. Clearly, $|\I| = N + O(M)$. Set
    $$
    \A := \{1!, 2!, \dots, (2N)!\} \cup \{(p - 2N)!, \dots, (p-2)!, (p-1)!\} \mod p.
    $$
Clearly, $\A\A \subseteq \A(p)\A(p)$, and from now on we work with $\A\A$.

From Wilson's theorem it follows, that $y! (p - 1 - y)! = (-1)^{y+1} \mod p$. Therefore, $y$ being odd implies $1/(p-1-y)! = y! \mod p$. Let $j \leqslant M$. Then
$$
\A\A \supseteq \{(y+j)!(p - 1 - y)!\ |\ y + j < 2N, y \text{ is odd}\} = 
$$
$$
= \{(y+j)!/y!\ |\ y + j < 2N, y \text{ is odd}\} = \{P_j(y)\ |\ y + j < 2N, y \text{ is odd}\}.
$$
%From here,
%$$
%\A\A 
%\supseteq \{x! (p - 1 - y)!\ :\ y < x < 2N, y \text{ is odd}\} =
%$$
%$$
%= \{x!/y!\ :\ y<x<2N, y \text{ is odd}\} 
%= \{P_{x - y}(y)\ :\ y < x < 2N, y \text{ is odd}\}.
%$$ 
This implies $Y_j \subseteq \A\A$ for all $j \leqslant M$. 

By equations \ref{P_image} and \ref{PP_image}, implied by the Lemmas \ref{lem:xj} and \ref{lem:xj_xk}, we obtain  the following (note, that $\delta \leqslant \eps_1, 1/2 - 2\eps_1 - 2\eps_2 - 2\kappa$ now plays a role):
$$
|Y_j| \geqslant N + O(Np^{-\delta}), \ \ 
|Y_k \cap Y_j| \leqslant N^2/p + O(N^2p^{-1-\delta}),\  k \neq j \text{ odd primes below $M$}.
$$

Set $A := \bigcup_{j}Y_j$ for prime $j \leqslant M$. We reduced the problem to show that $|A| \geqslant p + o(p)$.

Let us apply Lemma \ref{lem:set_intersect} with 
$$
a := N(1 + O(p^{-\delta})),\ 
b :=\frac{N^2}{p}(1 + O(p^{-\delta})),\ 
n \gg M/\log{M} \gg p^{\eps_2 - \kappa}
$$
Notice that by definition of $\delta$, which includes $\delta \leqslant \eps_2 - \eps_1 - \kappa$,  inequality $a/bn \ll p^{-\delta}$ holds, and therefore

$$
|A| \geqslant 
\frac{a^2}{b}\bigg( 1 - \frac{a}{bn}  \bigg) \geqslant 
p(1 + O(p^{-\delta})) = 
p + O(p^{1 - \delta}).
$$

Now our goal is to maximize $\delta$ subject to 
\begin{equation}
    \delta \leqslant 
    \begin{cases}
        \eps_1, \\
        1/2 - 2\eps_1 - 2\eps_2 - 2\kappa, \\
        \eps_2 - \eps_1 - \kappa.
    \end{cases}
\end{equation}
Solving this system, we obtain optimal parameters $\eps_1 := 1/14 - 4\kappa/7$, 
$\eps_2 := 1/7 - \kappa/7$, giving $\delta = 1/14 - 4\kappa/7$. This completes the proof.
\end{proof}

\section{On inequality $|\A_N/\A_N| \geqslant p + o(p)$}
We turn to the proof of Theorem \ref{thm:n-fact}.
\begin{proof}

Let $\I := \{L+1, \dots, L+N-M\}$, and $X_j := P_j(\I), j \leqslant M$,  with parameters $N, M$ depending on the case:

\textit{Case 1:  $N \gg p^{13/14}(\log{p})^{4/7}$.}

For this case one can apply the same argument as in the proof of Theorem \ref{thm:product} to obtain the desired bound. 

\medskip

\textit{Case 2: $p^{13/14}(\log{p})^{4/7} \gg N \gg p^{7/8}\log{p}$.}

Same as in the proof above, we write $N = p^{1 - \eps_1}$ and set $M = \floor{p^{\eps_2}}$ for $\eps_2 > 0$. Observe, that now $\eps_1$ is fixed, but $\eps_2$ is not.

Arguing as before, we obtain $|\A_N/\A_N| \geqslant p + O(p^{1 - \delta})$, where 
\begin{equation}
    \delta \leqslant 
    \begin{cases}
        \eps_1, \\
        1/2 - 2\eps_1 - 2\eps_2 - 2\kappa, \\
        \eps_2 - \eps_1 - \kappa.
    \end{cases}
\end{equation}

Let us set $\eps_2 := 1/6 - \eps_1/3 - \kappa/3$. Observe that $\eps_2 > 0$ since $\eps_1 \leqslant 1/2 - \kappa$. 
From here we obtain, that $\delta = \min(\eps_1, 1/6 - 4\eps_1/3 - 4\kappa/3) = 1/6 - 4\eps_1/3 - 4\kappa/3$ works. Notice, that $\delta > 0$ as long as $\eps_1 < 1/8 - \kappa$. 

This concludes the proof in case $N \gg p^{7/8}\log{p}$.

\medskip

\textit{Case 3: $p^{7/8}\log{p} \gg N \gg p^{4/5}(\log{p})^{8/5}$.}

Let $R$ be a positive integer we choose later. 
Let $M$ be a number with exactly $R$ odd primes below it. Clearly, $M \approx R\log{R}$. 

Clearly, for odd prime $j$ below $M$ we have 
$|X_j| \geqslant N + O(N^2p^{-1} + j^2\sqrt{p}(\log{p})^2) \gg N$ if $M^2 \ll Q$.

Clearly, summing $|X_k \cap X_j|$ for odd primes $k$ below odd prime $j \leqslant M$, we have
$$
\sum_{k < j}|X_k \cap X_j| \ll \frac{N^2}{p}R + RM^2\sqrt{p}(\log{p})^2 \ll N \; \text{ if } R \ll K, R^3(\log{R})^2 \ll Q.
$$

Therefore, setting $R := Q^{1/3}(\log{Q})^{-2/3}$, we obtain 
$$
|\A_N/\A_N| \geqslant 
\underbrace{|X_3 \cup 
X_5 \cup \dots|}_{\text{first $R$ odd primes}} - 
\sum_{k < j, \text{odd primes}}|X_k \cap X_j|
\gg
\underbrace{|X_3| + |X_5| + \dots}_{\text{first $R$ odd primes}} \gg NR,
$$
which completes the proof in this case.

\medskip

\textit{Case 4: $p^{4/5}(\log{p})^{8/5} \gg N \gg p^{1/2}(\log{p})^{2}$.}

 We follow the same line of argumentation, as in the \cite{garaev-hernandez}, but with modified bounds on sets $X_j$ and their intersections. 

From now on we work with all $j$, not just prime ones. Clearly, $J(j), J(k, j) \leqslant pj$, and therefore estimates 
$$
J_N(j), J_N(k, j) \leqslant \frac{N^2}{p^2}pj + O(j^2\sqrt{p}(\log{p})^2)
$$
hold, same as in \cite{garaev-hernandez}.

Same as in the proof of Lemma \ref{lem:xj}, we apply Cauchy-Bunyakovskii-Shwarz inequality:
$$
\#\{(x, y) : P_j(x) = P_j(y), 1 \leqslant x, y \leqslant N-M\}|X_j| \geqslant (N-M)^2,
$$
from where we obtain 
$$
|X_j| 
\geqslant 
\frac{N^2}{N + J_N(j)} 
\geqslant 
N + O\bigg(\frac{N^2j}{p} + j^2\sqrt{p}(\log{p})^2 \bigg)\ \  
\forall\ j \leqslant M.
$$

For $X_k \cap X_j$ we have the bound 
$$|X_k \cap X_j| \leqslant J_N(k, j) \leqslant \frac{N^2}{p}j + O(j^2\sqrt{p}(\log{p})^2) \ \ \forall\ k < j \leqslant M,
$$
same as in \cite{garaev-hernandez}.

Clearly, we have $|X_j| \gg N$ as long as $M \ll K, M^2 \ll Q$.

Clearly, we have $\sum_{k < j}|X_k \cap X_j| \ll N \ll |X_j|$ as long as $M^2 \ll K, M^3 \ll Q$. 

Therefore, similarly to \cite{garaev-hernandez}, we conclude
$$
| \A_N/\A_N|
\geqslant 
\sum_{j \leqslant M}\bigg(  |X_j|   - \sum_{k < j}|X_k \cap X_j|\bigg) \gg
\sum_{j \leqslant M}|X_j| \gg MN,
$$
where we set $M := \min(\sqrt{K}, \sqrt[3]{Q})$, which gives the desired bound.
\end{proof}

\subsection*{Acknowledgments}

A. Grebennikov is supported by Ministry of Science and Higher Education of the Russian Federation, agreement № 075–15–2019–1619. A. Sagdeev is supported in part by
ERC Advanced Grant ‘GeoScape’. He is also a winner of Young Russian Mathematics Contest and would like to thank its sponsors and jury. A. Semchankau was supported by the Foundation for the Advancement of Theoretical Physics and Mathematics “BASIS”. Also, A. Semchankau was partially supported by the University of Bordeaux Pause Program, the ANR project JINVARIANT, and Journal de Théorie des Nombres de Bordeaux. Besides that, he is very happy to thank the Institut de Mathématiques de Bordeaux for their hospitality and excellent working conditions. 
A significant part of this project was done during the research workshop ``Open problems in Combinatorics and Geometry III'', held in Adygea in October 2021.

\begin{comment}
\printbibliography[
heading=bibintoc,
] %Prints the entire bibliography with the titel "Whole bibliography"
\end{comment}

%\begin{comment}

%\end{comment}

\noindent{A.~Grebennikov\\
Saint-Petersburg State University, Saint-Petersburg, Russia \\
IMPA, Rio de Janeiro, Brazil}\\
{\tt sagresash@yandex.ru}

\vspace{4mm}

\noindent{A.~Sagdeev\\
Alfréd Rényi Institute of Mathematics, Budapest, Hungary}\\
{\tt sagdeevarsenii@gmail.com}

\vspace{4mm}

\noindent{A.~Semchankau\\
Moscow State University, Moscow, Russia\\
Saint-Petersburg State University, Saint-Petersburg, Russia}\\
{\tt aliaksei.semchankau@gmail.com}

\vspace{4mm}

\noindent{A.~Vasilevskii\\
Carnegie Mellon University, Pittsburgh, US}\\
{\tt avasileu@andrew.cmu.edu}

\end{document}